\documentclass[10pt]{article}%
\usepackage{amsmath}
\usepackage{amsfonts}
\usepackage{mathrsfs}
\usepackage{amssymb, color}
\usepackage[linkcolor=black,anchorcolor=black,citecolor=black]{hyperref}
\usepackage{graphicx}
\usepackage[body={15.5cm,21cm}, top=3cm]{geometry}%
\providecommand{\U}[1]{\protect\rule{.1in}{.1in}}
\providecommand{\U}[1]{\protect \rule{.1in}{.1in}}
\newtheorem{theorem}{Theorem}[section]

\newtheorem{corollary}[theorem]{Corollary}

\newtheorem{definition}[theorem]{Definition}

\newtheorem{lemma}[theorem]{Lemma}

\newtheorem{proposition}[theorem]{Proposition}
\newtheorem{remark}[theorem]{Remark}

\newenvironment{proof}[1][Proof]{\noindent \textbf{#1.} }{\  \rule{0.5em}{0.5em}}
\DeclareMathOperator*{\esssup}{ess\,sup}
\DeclareMathOperator*{\essinf}{ess\,inf}
\begin{document}
\title{Supermartingale Decomposition Theorem under $G$-expectation}
\author{ Hanwu Li\thanks{School of Mathematics, Shandong University,
lihanwu11@163.com.}
\and Shige Peng\thanks{School of Mathematics and Qilu Institute of Finance, Shandong University,
peng@sdu.edu.cn. Li and Peng's research was
partially supported by NSF (No. 10921101) and by the 111
Project (No. B12023).}
\and Yongsheng Song\thanks{Academy of Mathematics and Systems Science, CAS, Beijing, China, yssong@amss.ac.cn.
 Research supported  by NCMIS;
Key Project of NSF (No. 11231005); Key Lab of Random Complex
Structures and Data Science, CAS (No. 2008DP173182).}}
\maketitle
\begin{abstract}
The objective of this paper is to establish the decomposition theorem for supermartingales under the  $G$-framework. We first introduce a $g$-nonlinear expectation via a kind of $G$-BSDE and the associated supermartingales. We have shown that this kind of supermartingales have the decomposition similar to the classical case. The main ideas are to apply the uniformly continuous property of $S_G^\beta(0,T)$, the representation of the solution to $G$-BSDE and the approximation method via penalization.
\end{abstract}

\textbf{Key words}: $G$-expectation, $\mathbb{\hat{E}}^{g}$-supermartingale,  $\mathbb{\hat{E}}^{g}$-supermartingale decomposition theorem

\textbf{MSC-classification}: 60H10, 60H30
\section{Introduction}
The classical Doob-Meyer decomposition theorem  tells us that a large class of submartingales can be uniquely represented as the summation of a martingale and a predictable  increasing process.
This is one of fundamental results in the theory of stochastic analysis. This theorem was firstly proved in \cite{D} for the discrete time case. Then \cite{M62,M63} proved this result for the continuous time case. This theorem is important for the optimal stopping problem used to solve the pricing for the American options (see \cite{Ben},\cite{Karatzas}). It can be applied to study the problem of hedging contingent claims by portfolios constraint to take values in a given closed, convex set (see \cite{CK}). A general case of Doob-Meyer decomposition theorem was introduced in \cite{P99} when the supermartingale $Y_\cdot$ is defined by a nonlinear operator. It was proved that the nonlinear version of Doob-Meyer decomposition theorem also holds.

The objective of this paper is to solve the problem of decomposition
theorem of Doob-Meyer's type for a nonlinear supermartingale defined in a sublinear
expectation space-- $G$-expectation (upper case $G$). In order to understand
the motivation of this objective, let us recall its special linear case,
namely, in a framework of Wiener probability space $(\Omega,{\mathcal{F}},({\mathcal{F}%
})_{t\geq0},P)$ in which the canonical process $B_{t}(\omega)=\omega(t)$ for
$\omega \in \Omega=C_{0}([0,\infty))$ is a $d$-dimensional standard Brownian
motion. Given a function $g=g(s,\omega,y,z):[0,\infty)\times \Omega
\times \mathbb{R}\times \mathbb{R}^d\rightarrow \mathbb{R}$ where
$g(\cdot,y,z)$ satisfies the \textquotedblleft usual Lipschitz
conditions\textquotedblright \ in the framework of BSDE (see \cite{PP90}), such
 that, for each $T\in \lbrack0,\infty)$, the following BSDE has a
unique solution on $[0,T]$,
\[
y_{t}=\xi+\int_{t}^{T}g(s,y_{s},z_{s})ds+(A_{T}-A_{t})-\int_{t}^{T}z_{s}%
dB_{s},\, \, \,s\in \lbrack0,T],
\]
where $\xi$ is a given random variable in $L^{2}(\Omega,{\mathcal{F}}_{T},P)$
and $A_{\cdot}$ is a given continuous and increasing process with $A_{0}=0$
and $A_{t}\in L^{2}(\Omega,{\mathcal{F}}_{t},P)$ for each $t\in(0,T]$.  We call $y_{\cdot}$ a
$g$-supersolution. If $A_{\cdot}\equiv0$ then $y_{\cdot}$ is called a
$g$-solution. For the later case, since for each given $t\leq T$, the
${\mathcal{F}}_{t}$ measurable random variable $y_{t}$ is uniquely determined
by the terminal condition $y_{T}=\xi \in L^{2}(\Omega,{\mathcal{F}}_{T},P)$, we
then can define a backward semigroup \cite{P97,P03}
\[
{\mathcal{E}}_{t,T}^{g}[\xi]:=y_{t}
,\, \, \, \,0\leq t\leq T<\infty \tag{a}.
\]
This semiproup gives us a generalized notion of
nonlinear expectation with corresponding ${\mathcal{F}}_{t}$-conditional
expectation, called $g$-expectation \cite{P97}. By applying the comparison
theorem of BSDE we know that any $g$-supersolution $Y_{\cdot}$ is also a
$g$-supermartingale 
 (i.e., we have ${\mathcal{E}%
}_{s,t}^{g}[Y_{t}]\leq Y_{s}$, for each $s\leq t$). But the proof of the inverse claim,
namely, a $g$-supermartingale is a $g$-supersolution, is not at all trivial
(we refer to \cite{P99} for detailed proof). In fact this is a generalization
of the classical Doob-Meyer decomposition to the case of nonlinear
expectations, and the linear situation corresponds to the case $g\equiv0$.

Moreover, this nonlinear Doob-Meyer decomposition theorem plays a key role to
obtain the following  representation theorem of
nonlinear expectations: for a given arbitrary ${\mathcal{F}}_{t}$-conditional
nonlinear expectation $({\mathcal{E}}_{s,t}[\xi])_{0\leq s\leq t<\infty}$
with certain regularity, there exists a unique function $g=g(\cdot,y,z)$
satisfying the usual condition of BSDE, such that,
\[
{\mathcal{E}}_{t,T}[\xi]={\mathcal{E}}_{t,T}^{g}[\xi],\, \, \, \, \text{ for all
}\, \,0\leq t\leq T<\infty,\, \, \, \text{ and }\, \, \xi \in L^{2}(\Omega
,{\mathcal{F}}_{T},P).
\]
We refer to \cite{CHMP}, \cite{P03}, \cite{P05} for the proof of this very deep
result, also to \cite{DPR} a wide class of time consistent risk
measures are identified to be $g$-expectations.

It is known that volatility model uncertainty (VMU) involves essentially
non-dominated family of probability measures $\mathcal{P}$ on $(\Omega
,\mathcal{F})$. This is a main reason why many risk measures, and pricing operators cannot
be well-defined within a framework of probability space such as Wiener space $(\Omega
,{\mathcal{F}}_{T},P)$. \cite{P07} introduced the framework of (fully
nonlinear) time consistent $G$-expectation space $(\Omega,L_{G}^{1}%
(\Omega),\hat{\mathbb{E}})$ such that all probability measures in $\mathcal{P}$
are dominated by this sublinear expectation and such that the canonical
process $B_{\cdot}(\omega)=\omega(\cdot)$ becomes a nonlinear Brownian motion,
called $G$-Brownian. Many random variables, negligible under the probability
measure $P\in \mathcal{P}$, as well as under other measures in $\mathcal{P}$,
can be clearly distinguished in this new framework. The corresponding theory
of stochastic integration and stochastic calculus of It\^{o}'s type have been
established in \cite{P07,P10}. In particular, the existence and uniqueness of
BSDE driven by $G$-Brownian motion ($G$-BSDE) have been established in \cite{HJPS}. Roughly
speaking (see next section for details), a $G$-BSDE is as follows
\[
y_{t}=\xi+\int_{t}^{T}g(s,y_{s},z_{s})ds-\int_{t}^{T}z_{s}dB_{s}-(K_{T}%
-K_{t}),\quad t\in \lbrack0,T],
\]
where $g(\cdot,y,z)$ and $\xi$ satisfy very similar conditions with the classical case. The solution
of this $G$-BSDE consists of a triplet of adapted processes $(y_{\cdot},z_{\cdot
},K_{\cdot})$ where $K_{\cdot}$ is a decreasing $G$-martingale with $K_{0}=0$.
We then call $y_{\cdot}$ a ${g}$-solution under $\hat{\mathbb{E}}$. From the
existence and uniqueness of the $G$-BSDE, we can also define $\hat{\mathbb{E}}_{t,T}^{g}[\xi]=y_{t}$ which forms a time consistent nonlinear expectation.
If $K_{\cdot}$ is just a decreasing process then we call $y_\cdot$ a ${g}$-supersolution under $\hat{\mathbb{E}}$.

By the comparison theorem of $G$-BSDE obtained in \cite{HJPS1}, we can prove that a ${g}
$-supersolution under $\mathbb{\hat{E}}$ $Y_\cdot$ is also an  $\mathbb{\hat{E}}^{g}$-supermartingale, i.e., we have $\hat{\mathbb{E}}^g_{s,t}[Y_t]\leq Y_s$, for each $s\leq t$. The  objective of this paper is to
prove its inverse property: a continuous  $\mathbb{\hat{E}}^{g}$-supermartingale $Y_{\cdot}$ is also a ${g}$-supersolution under $\mathbb{E}$
. Namely, $Y_\cdot$ can be written as
\[
Y_{t}=Y_{T}+\int_{t}^{T}g(s,Y_{s},Z_{s})ds-\int_{t}^{T}Z_{s}dB_{s}%
+(A_{T}-A_{t}),\quad t\in \lbrack0,T],
\]
where $A$ is a continuous increasing process.  A special case of this result
is when $g\equiv0$. In this case $Y_{\cdot}$ is a $G$-supermartingale 
and it can be decomposed into the following
\[
Y_{t}=Y_{0}+\int_{0}^{t}Z_{s}dB_{s}-A_{t},%
\]
where $A$ is an increasing process. This is still a new and
non-trivial result.

The proof of this decomposition theorem involves a penalization procedure,
\[
y_{s}^{n}=Y_{T}+\int_{t}^{T}g(s,y_{s}^{n},z_{s}^{n})ds-\int_{t}^{T}z^n_{s}%
dB_{s}-(K_{T}^{n}-K_{t}^{n})+(L_{T}^{n}-L_{t}^{n}),\quad t\in \lbrack0,T],
\]
for $n=1,2,\cdots$, where $L_{t}^{n}=n\int_{0}^{t}(Y_{s}-y_{s}^{n})ds$ and
$K_{\cdot}^{n}$ is an decreasing martingale. 
In order to prove that $y^{n}\uparrow
Y$, it is necessary to show that $y^{n}\leq Y$. A main problem is that
the corresponding Doob's optional sampling is still an open problem. We
overcome this difficulty by proving that, for each probability dominated by
$\mathcal{P}$,  we have $y^{n}\leq Y$. We also need to introduce some new
methods, see Lemma \ref{the3.7} and Lemma \ref{the3.8}, to prove the uniform convergence of $y^{n}$. Generally
speaking, the well-known
Fatou's lemma cannot be directly and automatically used in this sublinear expectation framework. Besides,  a bounded  subset  in  $M_G^\beta(0,T)$  does not imply weakly compactness. Many proofs become
more delicate and challenging.

We believe that the proof of our new decomposition theorem of Doob-Meyer's
type under $G$-framework will play a key role for understanding and solving
many important problem. It is a key step towards the understanding and
solving a general representation theorem of dynamically consistent nonlinear
expectations, as well as dynamic risk measures and pricing operators.

The paper is organized as follows. In Section 2, we set up
some notations and results as preliminaries for the later proofs. Section 3 is
devoted to the study of the so-called  $\mathbb{\hat{E}}^{g}$-supermartingales. The
representation theorem is established with detailed proofs. In Section 4, we present the
relationship between the $\hat{\mathbb{E}}^g$-supermartingales and the fully nonlinear parabolic PDEs.

\section{Preliminaries}
\subsection{$G$-expectation and $G$-It\^{o}'s calculus}
The main purpose of this section is to recall some basic notions and results of $G$-expectation, which are needed in the sequel. The readers may refer to \cite{HJPS}, \cite{HJPS1}, \cite{P08a}, \cite{P10} for more details.

\begin{definition}
\label{def2.1} Let $\Omega$ be a given set and let $\mathcal{H}$ be a vector
lattice of real valued functions defined on $\Omega$, namely $c\in \mathcal{H%
}$ for each constant $c$ and $|X|\in \mathcal{H}$ if $X\in \mathcal{H}$. $%
\mathcal{H}$ is considered as the space of random variables. A sublinear
expectation $\hat{\mathbb{E}}$ on $\mathcal{H}$ is a functional $
\hat {\mathbb{E}}:\mathcal{H}\rightarrow \mathbb{R}$ satisfying the following
properties: for all $X,Y\in \mathcal{H}$, we have

\begin{description}
\item[(a)] Monotonicity: If $X\geq Y$, then $\hat{\mathbb{E}}[X]\geq
\hat{\mathbb{E}}[Y]$;

\item[(b)] Constant preserving: $\hat{\mathbb{E}}[c]=c$;

\item[(c)] Sub-additivity: $\hat{\mathbb{E}}[X+Y]\leq \hat{\mathbb{E}}[X]+%
\hat{\mathbb{E}}[Y]$;

\item[(d)] Positive homogeneity: $\hat{\mathbb{E}}[\lambda X]=\lambda
\hat{\mathbb{E}}[X]$ for each $\lambda \geq0$.
\end{description}

The triple $(\Omega,\mathcal{H},\hat{\mathbb{E}})$ is called a
sublinear expectation space.  $X \in\mathcal{ H}$ is called a random
variable in $(\Omega,\mathcal{H},\hat{\mathbb{E}})$. We often call
$Y = (Y_1, \ldots, Y_d), Y_i \in\mathcal{ H}$ a $d$-dimensional
random vector in $(\Omega,\mathcal{H},\hat{\mathbb{E}})$.
\end{definition}

\begin{definition}
\label{def2.2} Let $X_{1}$ and $X_{2}$ be two $n$-dimensional random vectors
defined respectively in sublinear expectation spaces $(\Omega_{1}%
,\mathcal{H}_{1},\mathbb{\hat{E}}_{1})$ and $(\Omega_{2},\mathcal{H}%
_{2},\mathbb{\hat{E}}_{2})$. They are called identically distributed, denoted
by $X_{1}\overset{d}{=}X_{2}$, if $\mathbb{\hat{E}}_{1}[\varphi(X_{1}%
)]=\mathbb{\hat{E}}_{2}[\varphi(X_{2})]$, for all$\ \varphi\in C_{Lip}%
(\mathbb{R}^{n})$, where $C_{Lip}(\mathbb{R}^{n})$ is the space of real
continuous functions defined on $\mathbb{R}^{n}$ such that
\[
|\varphi(x)-\varphi(y)|\leq C|x-y|\ \text{\ for
all}\ x,y\in\mathbb{R}^{n},
\]
where $C$ depends only on $\varphi$.
\end{definition}

\begin{definition}
\label{def2.3} In a sublinear expectation space $(\Omega,\mathcal{H}%
,\mathbb{\hat{E}})$, a random vector $Y=(Y_{1},\cdot\cdot\cdot,Y_{n})$,
$Y_{i}\in\mathcal{H}$, is said to be independent of another random vector
$X=(X_{1},\cdot\cdot\cdot,X_{m})$, $X_{i}\in\mathcal{H}$ under $\mathbb{\hat
{E}}[\cdot]$, denoted by $Y\perp X$, if for every test function $\varphi\in
C_{Lip}(\mathbb{R}^{m}\times\mathbb{R}^{n})$ we have $\mathbb{\hat{E}%
}[\varphi(X,Y)]=\mathbb{\hat{E}}[\mathbb{\hat{E}}[\varphi(x,Y)]_{x=X}]$.
\end{definition}

\begin{definition}
\label{def2.4} ($G$-normal distribution) A $d$-dimensional random vector
$X=(X_{1},\cdot\cdot\cdot,X_{d})$ in a sublinear expectation space
$(\Omega,\mathcal{H},\mathbb{\hat{E}})$ is called $G$-normally distributed if
for each $a,b\geq0$ we have
\[
aX+b\bar{X}\overset{d}{=}\sqrt{a^{2}+b^{2}}X,
\]
where $\bar{X}$ is an independent copy of $X$, i.e., $\bar{X}\overset{d}{=}X$
and $\bar{X}\bot X$. Here the letter $G$ denotes the function
\[
G(A):=\frac{1}{2}\mathbb{\hat{E}}[\langle AX,X\rangle]:\mathbb{S}%
_{d}\rightarrow\mathbb{R},
\]
where $\mathbb{S}_{d}$ denotes the collection of $d\times d$ symmetric matrices.
\end{definition}

It is proved in \cite{P08a} that $X=(X_{1},\cdot\cdot\cdot,X_{d})$ is $G$-normally
distributed if and only if for each $\varphi\in C_{Lip}(\mathbb{R}^{d})$,
$u(t,x):=\mathbb{\hat{E}}[\varphi(x+\sqrt{t}X)]$, $(t,x)\in\lbrack
0,\infty)\times\mathbb{R}^{d}$, is the solution of the following fully nonlinear parabolic equation:%
\[
\partial_{t}u-G(D_{x}^{2}u)=0,\ u(0,x)=\varphi(x).
\]
where  $D^2_xu=\{\partial_{x_ix_j}^2u\}_{i,j=1}^d$.

In the case $d=1$, the function $G:\mathbb{R}\rightarrow\mathbb{R}$ is a given  monotonic and sublinear function of the form
\begin{equation}\label{G}
G(a)=\frac{1}{2}(\bar{\sigma}^2a^+-{\underline{\sigma}}^2a^-), \,\,\, a\in {\mathbb{R}},
\end{equation}
where $\bar{\sigma}^2=\hat{\mathbb{E}}[X^2]$ and ${\underline{\sigma}}^2=-\hat{\mathbb{E}}[-X^2]$. In this paper we only consider the non-degenerate $G$-normal distribution, i.e., $\underline{\sigma}>0$ in the 1-dimensional case.

%

We present the notion of $G$-Brownian motion in a sublinear expectation space. For notational simplification, we only consider the case of $1$-dimensional $G$-Brownian motion. But the
methods  of this paper can be directly applied to  $d$-dimensional situations.

Let $\Omega=C_{0}([0,\infty);\mathbb{R})$ be the space of
real valued continuous functions on $[0,\infty)$ with $\omega_{0}=0$ endowed
with the following distance
$$
\rho(\omega^1, \omega^2):=\sum^\infty_{N=1} 2^{-N} [(\max_
{t\in[0,N]} | \omega^1_t-\omega^2_t|)  \wedge 1],
$$
and $B_t(\omega)=\omega_t,\, t\geq 0$,  $\omega\in \Omega$ be the canonical
process. For each $T>0$, set $\Omega_T=\{\omega(\cdot\wedge T),\,\, \omega\in \Omega\}$.
We denote by $\mathcal{B}(\Omega)$ the collection of all Borel-measurable subsets of $\Omega$.

\begin{definition}
\label{def2.5} i)
Set
\begin{align*}
L_{ip}(\Omega_T):=&\{ \varphi(B_{t_{1}},...,B_{t_{n}}):n\geq1,t_{1}%
,...,t_{n}\in\lbrack0,T],\varphi\in C_{Lip}(\mathbb{R}^{n})\},\\
L_{ip}(\Omega):=&\bigcup_{T>0}L_{ip}(\Omega_T).
\end{align*}
Let $G:\mathbb{R}\rightarrow\mathbb{R}$ be a given monotonic and sublinear
function of the form (\ref{G}). $G$-expectation is a sublinear expectation defined on the space of the random variable
$(\Omega,L_{ip}(\Omega))$ in the following way: for each $X\in L_{ip}(\Omega)$ in  the form
 $X=\varphi(B_{t_{1}}-B_{t_{0}},B_{t_{2}}-B_{t_{1}},\cdots,B_{t_{m}}-B_{t_{m-1}})\in L_{ip}(\Omega)$, with  $t_0<t_1<\cdots<t_m$, we set
\[
\mathbb{\hat{E}}[X]=\mathbb{\tilde{E}}[\varphi(\sqrt{t_{1}-t_{0}}\xi_{1}%
,\cdot\cdot\cdot,\sqrt{t_{m}-t_{m-1}}\xi_{m})],
\]
  where $\xi_{1},\cdot\cdot\cdot,\xi_{n}$ are
identically distributed $1$-dimensional $G$-normally distributed random
vectors in a sublinear expectation space $(\tilde{\Omega},\tilde{\mathcal{H}%
},\mathbb{\tilde{E}})$ such that $\xi_{i+1}$ is independent of $(\xi_{1}%
,\cdot\cdot\cdot,\xi_{i})$ for every $i=1,\cdot\cdot\cdot,m-1$.

The canonical process $B_{t}(\omega)=\omega_t$, $t\geq 0$,  is called a
$G$-Brownian motion on the sublinear expectation space $(\Omega, L_{ip}(\Omega),\hat{\mathbb{E}}[\cdot])$

ii) Let us define the conditional $G$-expectation $\mathbb{\hat{E}}_{t}$ of
$\xi\in L_{ip}(\Omega_T)$ knowing $L_{ip}(\Omega_t)$, for $t\in
\lbrack0,T]$. Without loss of generality we can assume that $\xi$ has the
representation $\xi=\varphi(B_{t_{1}}-B_{t_{0}},B_{t_{2}}-B_{t_{1}},\cdot
\cdot\cdot,B_{t_{m}}-B_{t_{m-1}})$ with $t=t_{i}$, for some $1\leq i\leq m$,
and we put
\[
\mathbb{\hat{E}}_{t_{i}}[\varphi(B_{t_{1}}-B_{t_{0}},B_{t_{2}}-B_{t_{1}}%
,\cdot\cdot\cdot,B_{t_{m}}-B_{t_{m-1}})]
\]%
\[
=\tilde{\varphi}(B_{t_{1}}-B_{t_{0}},B_{t_{2}}-B_{t_{1}},\cdot\cdot
\cdot,B_{t_{i}}-B_{t_{i-1}}),
\]
where
\[
\tilde{\varphi}(x_{1},\cdot\cdot\cdot,x_{i})=\mathbb{\hat{E}}[\varphi
(x_{1},\cdot\cdot\cdot,x_{i},B_{t_{i+1}}-B_{t_{i}},\cdot\cdot\cdot,B_{t_{m}%
}-B_{t_{m-1}})].
\]

\end{definition}

Define $\Vert X\Vert_{L_G^p}=(\mathbb{\hat{E}}[|\xi|^{p}])^{1/p}$ for $X
\in L_{ip}(\Omega)$ and $p\geq1$. Then \textmd{for all}$\ t\in\lbrack
0,T]$, $\mathbb{\hat{E}}_{t}[\cdot]$ is a continuous mapping on $L_{ip}(\Omega_T)$
w.r.t. the norm $\Vert\cdot\Vert_{L_G^p}$. Therefore it can be
extended continuously to the completion $L_{G}^{p}(\Omega_{T})$ of
$L_{ip}(\Omega_T)$ under the norm $\Vert\cdot\Vert_{L_G^p}$. Denis et al. \cite{DHP11}
proved that the completions of $C_{b}(\Omega_{T})$ (the set of bounded
continuous function on $\Omega_{T}$)  under the norm $\Vert\cdot\Vert_{L_G^p}$ coincides  with
$L_{G}^{p}(\Omega_{T})$.


 Let $\pi_{t}^{N}=\{t_{0}^{N},\cdots,t_{N}^{N}%
\}$, $N=1,2,\cdots$, be a sequence of partitions of $[0,t]$ such that $\mu
(\pi_{t}^{N})=\max\{|t_{i+1}^{N}-t_{i}^{N}|:i=0,\cdots,N-1\} \rightarrow0$,
the quadratic variation process of $B$ is defined by%
\[
\langle B\rangle_{t}=\lim_{\mu(\pi_{t}^{N})\rightarrow0}%
\sum_{j=0}^{N-1}(B_{t_{j+1}^{N}}-B_{t_{j}^{N}}%
)^{2}.
\]

Let us denote the set of all probability
measures on $(\Omega_{T},\mathcal{B}(\Omega_{T}))$ by $\mathcal{M}_{1}(\Omega_{T})$.

\begin{theorem}
\label{the2.7} (\cite{DHP11,HP09}) There exists a tight set $\mathcal{P}\subset\mathcal{M}_{1}(\Omega_{T})$
 such that
\[
\mathbb{\hat{E}}[X]=\sup_{P\in\mathcal{P}}E_{P}[X]\ \ \text{for
\ all}\ X\in L_{ip}(\Omega_T).
\]
$\mathcal{P}$ is called a set that represents $\mathbb{\hat{E}}$.
\end{theorem}

Let $\mathcal{P}$ be a tight set that represents $\mathbb{\hat{E}}$.
For this $\mathcal{P}$, we define capacity%
\[
c(A):=\sup_{P\in\mathcal{P}}P(A),\ A\in\mathcal{B}(\Omega_{T}).
\]
The set $A\subset\Omega_{T}$ is said to be polar if $c(A)=0$. A property holds
\textquotedblleft quasi-surely\textquotedblright\ (q.s. for short) if it holds
outside a polar set. In the following, we do not distinguish two random
variables $X$ and $Y$ if $X=Y$ q.s.. 
\begin{remark}\label{rem1.1}
Let $(\Omega,\mathcal{F},P^0)$ be a probability space and $(W_t)_{t\geq 0}$ be a 1-dimensional Brownian motion under $P^0$. Let $\mathbb{F}=\{\mathcal{F}_t\}$ be the augmented filtration generated by $W$. \cite{DHP11} proved that
\begin{displaymath}
\mathcal{P}_M:=\{P_h|P_h=P^0\circ X^{-1}, X_t=\int_0^t h_sdW_s, h\in L_{\mathbb{F}}^2([0,T];[\underline{\sigma},\overline{\sigma}])\},
\end{displaymath}
is a set that represents $\hat{\mathbb{E}}$, 
where $L_{\mathbb{F}}^2([0,T];[\underline{\sigma},\overline{\sigma}])$
is the collection of $\mathbb{F}$-adapted measurable processes with values in $[\underline{\sigma},\overline{\sigma}]$.
\end{remark}

For $\xi\in L_{ip}(\Omega_T)$, let $\mathcal{E}(\xi)=\hat{\mathbb{E}}[\sup_{t\in[0,T]}\hat{\mathbb{E}}_t[\xi]]$, where $\hat{\mathbb{E}}$ is the $G$-expectation. For convenience, we call $\mathcal{E}$ the $G$-evaluation. For $p\geq 1$ and $\xi\in L_{ip}(\Omega_T)$, define $\|\xi\|_{p,\mathcal{E}}=[\mathcal{E}(|\xi|^p)]^{1/p}$. Let $L_{\mathcal{E}}^p(\Omega_T)$ denote the completion of $L_{ip}(\Omega_T)$ under $\|\cdot\|_{p,\mathcal{E}}$. We shall give an estimate between the two norms $\|\cdot\|_{L_G^p}$ and $\|\cdot\|_{p,\mathcal{E}}$.


\begin{theorem}[\cite{S11}]\label{the2.4}
For any $\alpha\geq 1$ and $\delta>0$, $L_G^{\alpha+\delta}(\Omega_T)\subset L_{\mathcal{E}}^{\alpha}(\Omega_T)$. More precisely, for any $1<\gamma<\beta:=(\alpha+\delta)/\alpha$, $\gamma\leq 2$, we have
\begin{displaymath}
\|\xi\|_{\alpha,\mathcal{E}}^{\alpha}\leq \gamma^*\{\|\xi\|_{L_G^{\alpha+\delta}}^{\alpha}+14^{1/\gamma}
C_{\beta/\gamma}\|\xi\|_{L_G^{\alpha+\delta}}^{(\alpha+\delta)/\gamma}\},\forall \xi\in L_{ip}(\Omega_T),
\end{displaymath}
where $C_{\beta/\gamma}=\sum_{i=1}^\infty i^{-\beta/\gamma}$, $\gamma^*=\gamma/(\gamma-1)$.
\end{theorem}

 Independently, \cite {STZ11} proved $L_G^{\alpha}(\Omega_T)\subset L_{\mathcal{E}}^{2}(\Omega_T)$ for $\alpha>2$.

\begin{definition}
\label{def2.6} Let $M_{G}^{0}(0,T)$ be the collection of processes in the
following form: for a given partition $\{t_{0},\cdot\cdot\cdot,t_{N}\}=\pi
_{T}$ of $[0,T]$,
\[
\eta_{t}(\omega)=\sum_{j=0}^{N-1}\xi_{j}(\omega)\mathbf{1}_{[t_{j},t_{j+1})}(t),
\]
where $\xi_{i}\in L_{ip}(\Omega_{t_{i}})$, $i=0,1,2,\cdot\cdot\cdot,N-1$. For each
$p\geq1$ and $\eta\in M_G^0(0,T)$,  we denote by
\[
\|\eta\|_{H_G^p}=\{\hat{\mathbb{E}}[(\int_0^T|\eta_s|^2ds)^{p/2}]\}^{1/p},\,\,\,\,  \,\,
\Vert\eta\Vert_{M_{G}^{p}}:=({\hat{\mathbb{E}}}[\int_{0}^{T}|\eta_{s}|^{p}ds])^{1/p}.
\]
We use $H_G^p(0,T)$ and  $M_{G}^{p}(0,T)$  to denote
 the completion
of $M_{G}^{0}(0,T)$ under norms $\|\cdot\|_{H_G^p}$ and  $\|\cdot\|_{M_G^p}$ respectively.
\end{definition}

For two processes $ \eta\in M_{G}^{2}(0,T)$ and $ \xi\in M_{G}^{1}(0,T)$,
the $G$-It\^{o} integrals  $(\int^{t}_0\eta_sdB_s)_{0\leq t\leq T}$ and $(\int^{t}_0\xi_sd\langle
B\rangle_s)_{0\leq t\leq T}$  are well defined (see  Li-Peng \cite{LP} and Peng \cite{P10}).  Moreover, by Proposition 2.10 in \cite{LP} and the classical Burkholder-Davis-Gundy inequality, the following property holds.
\begin{proposition}\label{the1.3}
If $\eta\in H_G^{\alpha}(0,T)$ with $\alpha\geq 1$ and $p\in(0,\alpha]$, then we can get
$\sup_{u\in[t,T]}|\int_t^u\eta_s dB_s|^p\in L_G^1(\Omega_T)$ and
\begin{displaymath}
\underline{\sigma}^p c_p\hat{\mathbb{E}}_t[(\int_t^T |\eta_s|^2ds)^{p/2}]\leq
\hat{\mathbb{E}}_t[\sup_{u\in[t,T]}|\int_t^u\eta_s dB_s|^p]\leq
\bar{\sigma}^p C_p\hat{\mathbb{E}}_t[(\int_t^T |\eta_s|^2ds)^{p/2}].
\end{displaymath}
\end{proposition}

Let $S_G^0(0,T)=\{h(t,B_{t_1\wedge t}, \ldots,B_{t_n\wedge t}):t_1,\ldots,t_n\in[0,T],h\in C_{b,Lip}(\mathbb{R}^{n+1})\}$. For $p\geq 1$ and $\eta\in S_G^0(0,T)$, set $\|\eta\|_{S_G^p}=\{\hat{\mathbb{E}}[\sup_{t\in[0,T]}|\eta_t|^p]\}^{1/p}$. Let $S_G^p(0,T)$ denote the completion of $S_G^0(0,T)$ under the norm $\|\cdot\|_{S_G^p}$.

We consider the following type of $G$-BSDEs 
\begin{equation}\label{eq1}
Y_t=\xi+\int_t^T g(s,Y_s,Z_s)ds+\int_t^T f(s,Y_s,Z_s)d\langle B\rangle_s-\int_t^T Z_s dB_s-(K_T-K_t),
\end{equation}
where $g$ and $f$ are given functions
\begin{displaymath}
g(t,\omega,y,z),f(t,\omega,y,z):[0,T]\times\Omega_T\times\mathbb{R}\times\mathbb{R}\rightarrow \mathbb{R}
\end{displaymath}
satisfy the following properties:
\begin{description}
\item[(H1)] There exists some $\beta>1$ such that for any $y,z$, $g(\cdot,\cdot,y,z),f(\cdot,\cdot,y,z)\in M_G^{\beta}(0,T)$;
\item[(H2)] There exists some $L>0$ such that
\begin{displaymath}
|g(t,y,z)-g(t,y',z')|+|f(t,y,z)-f(t,y',z')|\leq L(|y-y'|+|z-z'|).
\end{displaymath}
\end{description}

For simplicity, we denote by $\mathfrak{S}_G^{\alpha}(0,T)$ the collection of processes $(Y,Z,K)$ such that $Y\in S_G^{\alpha}(0,T)$, $Z\in H_G^{\alpha}(0,T)$, $K$ is a decreasing $G$-martingale with $K_0=0$ and $K_T\in L_G^{\alpha}(\Omega_T)$.

\begin{definition}
Let $\xi\in L_G^{\beta}(\Omega_T)$ and $g$ and $f$ satisfy (H1) and (H2) for some $\beta>1$. A triplet of processes $(Y,Z,K)$ is called a solution of equation \eqref{eq1} if for some $1<\alpha\leq \beta$ the following properties hold:
\begin{description}
\item[(a)]$(Y,Z,K)\in\mathfrak{S}_G^{\alpha}(0,T)$;
\item[(b)]$Y_t=\xi+\int_t^T g(s,Y_s,Z_s)ds+\int_t^T f(s,Y_s,Z_s)d\langle B\rangle_s-\int_t^T Z_s dB_s-(K_T-K_t)$.
\end{description}
\end{definition}

\begin{theorem}[\cite{HJPS}]\label{the2.7}
Assume that $\xi\in L_G^{\beta}(\Omega_T)$ and $g$, $f$ satisfy (H1) and (H2) for some $\beta>1$. Then equation \eqref{eq1} has a unique solution $(Y,Z,K)$. Moreover, for any $1<\alpha< \beta$ we have $Y\in S_G^{\alpha}(0,T)$, $Z\in H_G^{\alpha}(0,T)$ and $K_T\in L_G^{\alpha}(\Omega_T)$.
\end{theorem}

We also have the comparison theorem for $G$-BSDE.
\begin{theorem}[\cite{HJPS1}]\label{the2.8}
Let $(Y_t^i,Z_t^i,K_t^i)_{t\leq T}$, $i=1,2$, be the solutions of the following two $G$-BSDEs:
\begin{displaymath}
Y^i_t=\xi^i+\int_t^T g_i(s)ds+\int_t^T f_i(s)d\langle B\rangle_s
+V_T^i-V^i_t-\int_t^T Z^i_s dB_s-(K^i_T-K^i_t),
\end{displaymath}
where $g_i(s)=g_i(s,Y^i_s,Z^i_s)$, $f_i(s)=f_i(s,Y^i_s,Z^i_s)$, $\xi^i\in L_G^{\beta}(\Omega_T)$, $\{V_t^i\}_{t\in[0,T]}$ are RCLL processes such that $\hat{\mathbb{E}}[\sup_{t\in[0,T]}|V_t^i|^\beta]<\infty$, $g_i,f_i$  satisfy (H1) and (H2) with $\beta>1$. Assume that $\xi^1\geq \xi^2$, $f_1\geq f_2$, $g_1\geq g_2$ and $\{V^1_t-V^2_t\}$ is a nondecreasing process, then $Y_t^1\geq Y_t^2$.
\end{theorem}








\subsection{Some results of classical penalized BSDEs}

In this subsection, we will introduce some notions and results following Peng \cite{P99}. The probability space and filtration is given in Remark \ref{rem1.1}. 
 For a given stopping time $\tau$, we now consider the following classical BSDE:
\begin{equation}\label{eq2}
y_t=\xi+\int_{t\wedge\tau}^\tau g(s,y_s,z_s)ds+(A_\tau-A_{t\wedge\tau})-\int_{t\wedge\tau}^\tau z_sdW_s,
\end{equation}
where $\xi\in L^2(\Omega,\mathcal{F}_\tau)$ and $g$ satisfies the following conditions:

 \begin{description}
\item [(A1)] $g(\cdot,y,z)\in L^2_{\mathbb{F}}(0,T;\mathbb{R})$, for each $(y,z)\in \mathbb{R}^{2}$;
\item [(A2)] There exists a constant $L>0$ such that
 \begin{displaymath}
|g(t,y,z)-g(t,y',z')|\leq L(|y-y'|+|z-z'|).
\end{displaymath}
 \end{description}

 Here $A$ is a given RCLL increasing process with $A_0=0$ and $E[A_\tau^2]<\infty$. We call $(y_t)$ the $g$-supersulotion on $[0,\tau]$ if $(y,z)$ solves \eqref{eq2}. In particular, when $A\equiv 0$, $(y_t)$ is called a $g$-solution on $[0,\tau]$.

\begin{definition}
An $\mathcal{F}_t$-progressively measurable real-valued process $(Y_t)$ is called a $g$-supermartingale on $[0,T]$ in strong sense if, for each stopping time $\tau\leq T$, $E[|Y_\tau|^2]<\infty$, and the $g$-solution $(y_t)$ on $[0,\tau]$ with terminal condition $y_\tau=Y_\tau$, satisfies $y_\sigma\leq Y_\sigma$ for all stopping time $\sigma\leq \tau$.
\end{definition}

\begin{definition}
An $\mathcal{F}_t$-progressively measurable real-valued process $(Y_t)$ is called a $g$-supermartingale on $[0,T]$ in weak sense if, for each deterministic time $ t\leq T$, $E[|Y_t|^2]<\infty$, and the $g$-solution $(y_t)$ on $[0,t]$ with terminal condition $y_t=Y_t$, satisfies $y_s\leq Y_s$ for all deterministic time $s\leq t$.
\end{definition}

It is obvious that a $g$-supermartingale in strong sense is also a $g$-supermartingale in weak sense. \cite{CP} proved that, under assumptions similar to the classical case, a $g$-supermartingale in weak sense coincides with a $g$-supermartingale in strong sense. This result is a generalization of the classical Optional Stopping Theorem. If $(Y_t)$ is a $g$-supersolution on $[0,T]$, it follows from the comparison theorem that $(Y_t)$ is a $g$-supermartingale. In fact, \cite{P99} proved that the inverse problem, i.e., nonlinear version of Doob-Meyer decomposition theorem, also holds. The method of proof is to apply the penalization approach and the first step is the following lemma.

\begin{lemma}[\cite{P99}]\label{the2.11}
Let $(Y_t)$ be a right-continuous $g$-supermartingale on $[0,T]$ in strong sense with $E[\sup_{0\leq t\leq T}|Y_t|^2]\leq \infty$. Assume that $g$ satisfies (A1) and (A2). For each $n=1,2,\cdots$, consider the following BSDEs:
 \begin{displaymath}
y_t^n=Y_T+\int_t^T g(s,y_s^n,z_s^n)ds+n\int_t^T(Y_s-y_s^n)ds-\int_t^T z_s^ndW_s.
\end{displaymath}
Then, for each $n=1,2,\cdots$, $Y_t\geq y_t^n$.
\end{lemma}

\begin{remark}\label{the2.12}
Set $M_t=\int_0^t h_sdW_s$, where $h\in L_{\mathbb{F}}^2([0,T];[\underline{\sigma},\overline{\sigma}])$. If the BSDE \eqref{eq2} is driven by $M$,
 \begin{displaymath}
y_t=\xi+\int_{t\wedge\tau}^\tau g(s,y_s,z_s)ds+(A_\tau-A_{t\wedge\tau})-\int_{t\wedge\tau}^\tau z_sdM_s,
\end{displaymath}
then, we can define a $g_M$-supersolution (also $g_M$-solution) and a $g_M$-supermartingale in strong sense (also in weak sense). Furthermore, we have a similar result as Lemma \ref{the2.11}.
\end{remark}

\section{Nonlinear expectations generated by $G$-BSDEs and the associated supermartingales}

For simplicity, we only consider the following $G$-BSDE driven by 1-dimensional $G$-Brownian motion. The result still holds for  multi-dimensional cases.
\begin{equation}\label{eq3}
Y_t^{T,\xi}=\xi+\int_t^T g(s,Y_s^{T,\xi},Z_s^{T,\xi})ds\\
-\int_t^T Z_s^{T,\xi} dB_s-(K_T^{T,\xi}-K_t^{T,\xi}),
\end{equation}
where $g$  satisfies the following conditions:
\begin{description}
\item[(H1')] There exists some $\beta>2$ such that for any $y,z$, $g(\cdot,\cdot,y,z)\in M_G^{\beta}(0,T)$;
\item[(H2)] There exists some $L>0$ such that
\begin{displaymath}
|g(t,y,z)-g(t,y',z')|\leq L(|y-y'|+|z-z'|).
\end{displaymath}
\end{description}

 For each $\xi\in L_G^{\beta}(\Omega_T)$ with $\beta>2$, we define
\begin{displaymath}
\hat{\mathbb{E}}^{g}_{t,T}[\xi]:=Y_t^{T,\xi}.
\end{displaymath}

\begin{definition}
A process $\{Y_t\}_{t\in[0,T]}$ 
is called an  $\mathbb{\hat{E}}^{g}$-supermartingale if, for each $t\leq T$, $Y_t\in L_G^\beta(\Omega_t)$ with $\beta>2$ and $\hat{\mathbb{E}}^g_{s,t}[Y_t]\leq Y_s$, $\forall 0\leq s\leq t\leq T$.
\end{definition}

\begin{remark}
(i)If $g=0$, the  $\mathbb{\hat{E}}^{g}$-supermartingale $\{Y_t\}_{t\in[0,T]}$ is in fact a $G$-supermartingale.\\
(ii)If the decreasing $G$-martingale $(K_t)$ in \eqref{eq3} is replaced by a continuous decreasing process $A$ with $A_0=0$, $\hat{\mathbb{E}}[A_T^2]<\infty$, then $(Y_t)$ is called a ${g}$-supersolution under $\mathbb{\hat{E}}$ on $[0,T]$. It follows from the comparison theorem of $G$-BSDE  that a ${g}$-supersolution under $\hat{\mathbb{E}}$ is also an  $\mathbb{\hat{E}}^{g}$-supermaringale.\\
(iii)If there exists a generator $f$ corresponding to the $d\langle B\rangle$ term in \eqref{eq3}, we can define the operator $\mathbb{\hat{E}}^{g,f}_{t,T}[\cdot]$ and the associated  $\mathbb{\hat{E}}^{g,f}$-supermartingales.
\end{remark}

The following theorem, which is a main result of this paper, tells us  that an  $\mathbb{\hat{E}}^{g}$-supermartingale is also a $g$-supersolution under $\hat{\mathbb{E}}$. It generalizes the well-known decomposition theorem of Doob-Meyer's type to a framework of fully nonlinear expectation--G-expectation.



\begin{theorem}\label{the3.3}
Let $Y=(Y_t)_{t\in[0,T]}\in S_G^\beta(0,T)$ be an  $\mathbb{\hat{E}}^{g}$-supermartingale with $\beta>2$. Suppose that $g$ satisfies (H1') and (H2). Then $(Y_t)$ has the following decomposition
\begin{equation}\label{eq4}
Y_t=Y_0-\int_0^t g(s,Y_s,Z_s) ds+\int_0^t Z_s dB_s-A_t,\quad \textrm{q.s.},
\end{equation}
where $\{Z_t\}\in M_G^2(0,T)$ and $\{A_t\}$ is a continuous nondecreasing process with $A_0=0$ and $A_T\in L_G^2(\Omega_T)$. Furthermore, the above decomposition is unique.
\end{theorem}




We divide the proof into a sequence of lemmas. For $P\in\mathcal{P}_M$, $\mathbb{F}$-stopping time $\tau$, and $\mathcal{F}_{\tau}$-measurable random variable $\eta\in L^2(P)$, let $(Y^P,Z^P)$ denote the solution to the following standard BSDE:
\begin{displaymath}
Y^P_s=\eta+\int_s^\tau g(r,Y^P_r,Z^P_r)dr -\int_s^\tau Z^P_r dB_r,\quad 0\leq s\leq \tau, \quad P\textrm{-}a.s..
\end{displaymath}

We recall from \cite{STZ12} that every $P\in\mathcal{P}_M$ satisfies the martingale representation property. Then there exists a unique adapted solution $(Y^P,Z^P)$ of the above equation. We define $E^{g,P}_{t,\tau}[\eta]:=Y_t^P$. For $P\in\mathcal{P}_M$ and $t\in [0,T]$, set $\mathcal{P}(t,P):=\{Q\in\mathcal{P}_M\big| Q|_{\mathcal{F}_t}=P|_{\mathcal{F}_t}\}$. The following lemma provides a representation for solution $Y^{T,\xi}$ of equation \eqref{eq3}.

\begin{lemma}[\cite{STZ12}]\label{the3.4}
For each $\xi\in L_G^\beta(\Omega_T)$ with $\beta>2$, we have, for $P\in\mathcal{P}_M$ and $t\in [0,T]$
\begin{displaymath}
\hat{\mathbb{E}}^g_{t,T}[\xi]={\esssup_{Q\in \mathcal{P}(t,P)}}^P E^{g,Q}_{t,T}[\xi], \textrm{ P-a.s.}.
\end{displaymath}
\end{lemma}

For reader's convenience, we give a brief proof here.

\begin{proof}
By the comparison theorem of classical BSDE, for $Q\in\mathcal{P}(t,P)$, we have $E^{g,Q}_{t,T}[\xi]\leq \hat{\mathbb{E}}^g_{t,T}[\xi]$, $Q$-$a.s.$. Consequently, we have $E^{g,Q}_{t,T}[\xi]\leq \hat{\mathbb{E}}^g_{t,T}[\xi]$, $P$-$a.s.$. Besides, by Theorem 16 in \cite{HP13} (see also Proposition 3.4 in \cite{STZ11}) and noting that $(K^{T,\xi}_t)$ is a decreasing $G$-martingale, we have \[0=\mathbb{\hat{E}}_t[K^{T,\xi}_T-K^{T,\xi}_t]=\underset{Q\in\overline{\mathcal{P}(t,P)}}{\esssup}^pE^Q_t[K_T^{T,\xi}-K_t^{T,\xi}],\ \  P\textrm{-}a.s.,\]
where $\overline{\mathcal{P}(t,P)}$ is the closure of $\mathcal{P}(t,P)$ with respect to the weak topology. Then there exists $Q\in \overline{\mathcal{P}(t,P)}$, such that $E^Q[K^{T,\xi}_T-K^{T,\xi}_t]=0$. Choose $\{Q_n\}\subset \mathcal{P}(t,P)$ such that $Q_n\rightarrow Q$ weakly, by Lemma 29 in \cite{DHP11}, then we obtain
\begin{displaymath}
E^{Q_n}[|K^{T,\xi}_T-K^{T,\xi}_t|^{1+\alpha}]\leq \{E^{Q_n}[|K^{T,\xi}_T-K^{T,\xi}_t|^{\frac{1}{1-\alpha}}]\}^{1-\alpha}\{E^{Q_n}[|K^{T,\xi}_T-K^{T,\xi}_t|]\}^\alpha\rightarrow0,
\end{displaymath}
where $0<\alpha<1-\frac{1}{\beta}$. By Proposition 3.2 in \cite{B}, we derive that
\begin{displaymath}
|\hat{\mathbb{E}}^g_{t,T}[\xi]-E^{g,Q_n}_{t,T}[\xi]|^{1+\alpha}\leq C_\alpha E_t^{Q_n}[|K^{T,\xi}_T-K^{T,\xi}_t|^{1+\alpha}], Q_n\textrm{-}a.s..
\end{displaymath}
Consequently, the above inequality holds $P$-$a.s.$. Then we have
\begin{displaymath}
E^P[|\hat{\mathbb{E}}^g_{t,T}[\xi]-E^{g,Q_n}_{t,T}[\xi]|^{1+\alpha}]\leq C_\alpha E^{Q_n}[|K^{T,\xi}_T-K^{T,\xi}_t|^{1+\alpha}]\rightarrow 0.
\end{displaymath}
The proof is complete.
\end{proof}

\begin{lemma}\label{the3.5}
Let $Y=(Y_t)_{t\in[0,T]}\in S_G^\beta(0,T)$ be an  $\mathbb{\hat{E}}^{g}$-supermartingale with $\beta>2$. 
Suppose that $g$ satisfies (H1') and (H2).  For each   $n=1,2,\cdots$,  consider the following $G$-BSDEs:
\begin{equation}\label{eq5}
y_t^n=Y_T+\int_t^T g(s,y_s^n,z_s^n)ds+n\int_t^T(Y_s-y_s^n)ds-\int_t^T z_s^ndB_s-(K_T^n-K_t^n).
\end{equation}
Then,  for $n=1,2,\cdots$, $Y_t\geq y_t^n$, $q.s.$.
\end{lemma}

\begin{proof}
Suppose the lemma were false. Then we could find some $t\in[0,T]$ and $P^*\in\mathcal{P}_M$ such that $P^*(y_t^n>Y_t)>0$. 


Applying Lemma \ref{the3.4} and the definition of  $\mathbb{\hat{E}}^{g}$-supermartingales, we have for any $P\in\mathcal{P}_M$ and $s\leq t$,
\begin{displaymath}
E^{g,P}_{s,t}[Y_t]\leq {\esssup_{P'\in\mathcal{P}(s, P)}}^P E^{g,P'}_{s,t}[Y_t]=\hat{\mathbb{E}}^g_{s,t}[Y_t]\leq Y_s, P\textrm{-}a.s..
\end{displaymath}
This shows that, under the measure $P\in\mathcal{P}_M$, $(Y_t)$ can be seen as an $g_B$-supermartingale in weak sense (see Remark \ref{the2.12}). Since $(Y_t)\in S_G^\beta(0,T)$ is continuous, it is an $g_B$-supermartingale in strong sense. For any $Q\in \mathcal{P}(t,P^*)$, let $(\bar{Y}^Q,\bar{Z}^Q)$ denote the solution to the following standard BSDE:
\begin{displaymath}
\bar{Y}^Q_s=Y_T+\int_s^T g_n(r,\bar{Y}^Q_r,\bar{Z}^Q_r)dr -\int_s^T \bar{Z}^Q_r dB_r, \quad Q\textrm{-}a.s..
\end{displaymath}
where $g_n(s,y,z)=f(s,y,z)+n(Y_s-y)$. Since $(Y_t)$ is an $g_B$-supermartingale and $g$ satisfies the assumptions in Lemma \ref{the2.11}, then it is easy to check that $Y_t\geq E^{g_n,Q}_{t,T}[Y_T](=\bar{Y}^Q_t)$, $Q$-$a.s.$. By the definition of $\mathcal{P}(t,P^*)$, we obtain that $Y_t\geq E^{g_n,Q}_{t,T}[Y_T]$, $P^*$-$a.s.$. Again by Lemma \ref{the3.4}, we have  $\esssup_{Q\in\mathcal{P}(t, P^*)}^{P^*} E^{g_n,Q}_{t,T}[Y_T]=y^n_t$, $P^*$-$a.s.$. This leads to a contradiction.
\end{proof}

It follows from the comparison theorem that $y_t^n\leq y_t^{n+1}$. Thus for all $n=1,2,\cdots$, $|y_t^n|$ is dominated by $|y_t^1|\vee|Y_t|$. Then we can find a constant $C$ independent of $n$, such that for $1<\alpha<\beta$, and for all $n=1,2,\cdots$,
\begin{equation}\label{eq6}
\hat{\mathbb{E}}[\sup_{t\in[0,T]}|y_t^n|^\alpha]\leq \hat{\mathbb{E}}[\sup_{t\in[0,T]}(|y_t^1|\vee|Y_t|)^{\alpha}]\leq C.
\end{equation}

Now let $L_t^n=n\int_0^t (Y_s-y_s^n)ds$, then $(L_t^n)_{t\in[0,T]}$ is an increasing process. We can rewrite $G$-BSDE \eqref{eq5} as
\begin{displaymath}
y_t^n=Y_T-\int_t^T z_s^ndB_s+\int_t^T g(s,y_s^n,z_s^n) ds-(K_T^n-K_t^n)+(L_T^n-L_t^n).
\end{displaymath}

\begin{lemma}\label{the3.6}
There exists a constant $C$ independent of $n$, such that for $1<\alpha<\beta$,
\begin{displaymath}
\hat{\mathbb{E}}[(\int_0^T (z_t^n)^2 dt)^{\frac{\alpha}{2}}]\leq C,\ \hat{\mathbb{E}}[|K_T^n|^\alpha]\leq C, \ \hat{\mathbb{E}}[|L_T^n|^\alpha]=n^\alpha\hat{\mathbb{E}}[(\int_0^T |Y_s-y_s^n|ds)^\alpha]\leq C.
\end{displaymath}
\end{lemma}
\begin{proof}
By a similar analysis as Proposition 3.5 in \cite{HJPS}, we have
\begin{align*}
\hat{\mathbb{E}}[(\int_0^T |z_s^n|^2ds)^{\frac{\alpha}{2}}]&\leq C_{\alpha}\{\hat{\mathbb{E}}[\sup_{t\in[0,T]}|y_t^n|^\alpha]
+(\hat{\mathbb{E}}[\sup_{t\in[0,T]}|y_t^n|^\alpha])^{\frac{1}{2}}(\hat{\mathbb{E}}[(\int_0^T|g(s,0,0)|ds)^\alpha])^{\frac{1}{2}}\},\\
\hat{\mathbb{E}}[|L_T^n-K_T^n|^\alpha]&\leq C_{\alpha}\{\hat{\mathbb{E}}[\sup_{t\in[0,T]}|y_t^n|^\alpha]+\hat{\mathbb{E}}[(\int_0^T|g(s,0,0)|ds)^\alpha]\},
\end{align*}
where the constant $C_\alpha$ depends on $\alpha,T,G$ and $L$. Thus we conclude that there exists a constant $C$ independent of $n$, such that for $1<\alpha<\beta$,
\begin{displaymath}
\hat{\mathbb{E}}[(\int_0^T |z_t^n|^2 dt)^{\frac{\alpha}{2}}]\leq C, \quad \hat{\mathbb{E}}[|L_T^n-K_T^n|^\alpha]\leq C.
\end{displaymath}
Since $L_T^n$ and $-K_T^n$ are nonnegative, we get
\begin{displaymath}
\hat{\mathbb{E}}[|K_T^n|^\alpha]\leq C, \quad \hat{\mathbb{E}}[|L_T^n|^\alpha]=n^\alpha\hat{\mathbb{E}}[(\int_0^T |Y_s-y_s^n|ds)^\alpha]\leq C.
\end{displaymath}
\end{proof}

For $1<\alpha<\beta$, we obtain the following inequality.
\begin{displaymath}
\begin{split}
&\quad n\hat{\mathbb{E}}[\int_0^T(Y_s-y_s^n)^{\alpha}ds]\\
&\leq Cn\hat{\mathbb{E}}[\int_0^T(Y_s-y^n_s)|Y_s|^{\alpha-1}ds]+Cn\hat{\mathbb{E}}[\int_0^T(Y_s-y_s^n)|y_s^n|^{\alpha-1}ds]\\
&\leq Cn(\hat{\mathbb{E}}[\sup_{s\in[0,T]}|Y_s|^{(\alpha-1)p}])^{1/p}(\hat{\mathbb{E}}[(\int_0^T(Y_s-y_s^n)ds)^q])^{1/q}\\
&+Cn(\hat{\mathbb{E}}[\sup_{s\in[0,T]}|y_s^n|^{(\alpha-1)p}])^{1/p}(\hat{\mathbb{E}}[(\int_0^T(Y_s-y_s^n)ds)^q])^{1/q},
\end{split}
\end{displaymath}
where $p,q>1$ satisfy $\frac{1}{p}+\frac{1}{q}=1$, $(\alpha-1)p<\beta$ and $q<\beta$. By estimate \eqref{eq6} and Lemma \ref{the3.6}, there exists a constant $C$ independent of $n$, such that
\begin{displaymath}
n\hat{\mathbb{E}}[\int_0^T(Y_s-y_s^n)^{\alpha}ds]\leq C.
\end{displaymath}
This implies that $y^n$ converges to $Y$ in $M_G^\alpha(0,T)$. In fact, this convergence holds in $S_G^\alpha(0,T)$. In order to prove this conclusion, we need the following uniformly continuous property for any $Y\in S_G^p(0,T)$ with $p>1$.

\begin{lemma}\label{the3.7}
For  $Y\in S_G^p(0,T)$ with $p>1$, we have, by setting $Y_s:=Y_T$ for $s>T$,
\begin{displaymath}
F(Y):=\limsup_{\varepsilon\rightarrow0}(\hat{\mathbb{E}}[\sup_{t\in[0,T]}\sup_{s\in[t,t+\varepsilon]}|Y_t-Y_s|^p])^\frac{1}{p}=0.
\end{displaymath}
\end{lemma}

\begin{proof}
For  $Y\in S_G^0(0,T)$, the conclusion is obvious. Noting that for $Y, Y'\in S^p_G(0,T)$ we have
\[|F(Y)-F(Y')|\leq C\|Y-Y'\|_{S^p_G}, \] which implies that $F(Y)=0$ for any  $Y\in S_G^p(0,T)$.
\end{proof}


\begin{lemma}\label{the3.8}
For some $1<\alpha<\beta$, we have
\begin{displaymath}
\lim_{n\rightarrow\infty}\hat{\mathbb{E}}[\sup_{t\in[0,T]}|Y_t-y_t^n|^\alpha]=0.
\end{displaymath}
\end{lemma}

\begin{proof}
By applying $G$-It\^{o}'s formula to $e^{-nt}y_t^n$, we get
\begin{displaymath}
y_t^n=e^{nt}\hat{\mathbb{E}}_t[e^{-nT}Y_T+\int_t^T ne^{-ns}Y_s ds+\int_t^T e^{-ns}g(s,y_s^n,z_s^n)ds].
\end{displaymath}
Then we obtain
\begin{displaymath}
0\leq Y_t-y_t^n\leq \hat{\mathbb{E}}_t[\tilde{Y}_t^n-\int_t^T e^{n(t-s)}g(s,y_s^n,z_s^n)ds],
\end{displaymath}
where $\tilde{Y}_t^n=e^{n(t-T)}(Y_t-Y_T)+\int_t^T ne^{n(t-s)}(Y_t-Y_s )ds$. By H\"{o}lder's inequality, it follows that
\begin{align*}
|\int_t^T e^{n(t-s)}g(s,y_s^n,z_s^n)ds|
\leq &\frac{1}{\sqrt{2n}}(\int_0^T g^2(s,y_s^n,z_s^n)ds)^{1/2}\\
\leq &\frac{C}{\sqrt{n}}(\sup_{s\in[0,T]}|y_s^n|^2+\int_0^T (g^2(s,0,0)+|z_s^n|^2)ds)^{1/2}.
\end{align*}
Then for $1<\alpha<\beta$, we have
\begin{equation}\label{eq7}
\hat{\mathbb{E}}[\sup_{t\in[0,T]}|\int_t^T e^{n(t-s)}g(s,y_s^n,z_s^n)ds|^\alpha]\rightarrow 0, \textrm{ as } n\rightarrow\infty.
\end{equation}

For $\varepsilon>0$, it is simple to show that
\begin{align*}
|\tilde{Y}_t^n|&=|e^{n(t-T)}(Y_t-Y_T)+\int_{t+\varepsilon}^T ne^{n(t-s)}(Y_t-Y_s )ds+\int_t^{t+\varepsilon} ne^{n(t-s)}(Y_t-Y_s )ds|\\
&\leq e^{n(t-T)}|Y_t-Y_T|+e^{-n\varepsilon}\sup_{s\in[t+\varepsilon,T]}|Y_t-Y_s|+\sup_{s\in[t,t+\varepsilon]}|Y_s-Y_t|.
\end{align*}
For $T>\delta>0$, from the above inequality we obtain
\begin{align*}
\sup_{t\in[0,T-\delta]}|\tilde{Y}_t^n|&\leq e^{-n\delta}\sup_{t\in[0,T-\delta]}|Y_t-Y_T|+e^{-n\varepsilon}\sup_{t\in[0,T-\delta]}\sup_{s\in[t+\varepsilon,T]}|Y_t-Y_s|+\sup_{t\in[0,T-\delta]}\sup_{s\in[t,t+\varepsilon]}|Y_s-Y_t|\\
&\leq 2\sup_{t\in[0,T]}|Y_t|(e^{-n\varepsilon}+e^{-n\delta})+\sup_{t\in[0,T]}\sup_{s\in[t,t+\varepsilon]}|Y_s-Y_t|.
\end{align*}
It is easy to check that for each fixed $\varepsilon,\delta>0$,
\begin{equation}\begin{split}\label{eq9}
\hat{\mathbb{E}}[\sup_{t\in[0,T-\delta]}|\tilde{Y}_t^n|^\beta]&\leq C[(e^{-n\beta\varepsilon}+e^{-n\beta\delta})\hat{\mathbb{E}}[\sup_{t\in[0,T]}|Y_t|^\beta]+\hat{\mathbb{E}}[\sup_{t\in[0,T]}\sup_{s\in[t,t+\varepsilon]}|Y_s-Y_t|^\beta]]\\
&\rightarrow C\hat{\mathbb{E}}[\sup_{t\in[0,T]}\sup_{s\in[t,t+\varepsilon]}|Y_s-Y_t|^\beta], \textrm{ as } n\rightarrow \infty.
\end{split}\end{equation}

For $1<\alpha<\beta$ and $0<\delta<T$, noting that $y_t^1\leq y_t^n\leq Y_t$, $n=1,2,\cdots$, we have
\begin{equation}\begin{split}\label{eq10}
&\quad\hat{\mathbb{E}}[\sup_{t\in[0,T]}|Y_t-y_t^n|^\alpha]\\
&\leq \hat{\mathbb{E}}[\sup_{t\in[0,T-\delta]}|Y_t-y_t^n|^\alpha]+\hat{\mathbb{E}}[\sup_{t\in[T-\delta,T]}|Y_t-y_t^n|^\alpha]\\
&\leq \hat{\mathbb{E}}[\sup_{t\in[0,T-\delta]}\hat{\mathbb{E}}_t[|\tilde{Y}_t^n-\int_t^T e^{n(t-s)}g(s,y_s^n,z_s^n)ds|^\alpha]]+\hat{\mathbb{E}}[\sup_{t\in[T-\delta,T]}|Y_t-y_t^1|^\alpha]\\
&\leq \hat{\mathbb{E}}[\sup_{t\in[0,T-\delta]}\hat{\mathbb{E}}_t[\sup_{u\in[0,T-\delta]}|\tilde{Y}_u^n|^\alpha]]
+\hat{\mathbb{E}}[\sup_{t\in[T-\delta,T]}|Y_t-y_t^1|^\alpha]\\
&\quad+\hat{\mathbb{E}}[\sup_{t\in[0,T-\delta]}\hat{\mathbb{E}}_t[\sup_{u\in[0,T]}|\int_u^T e^{n(t-s)}g(s,y_s^n,z_s^n)ds|^\alpha]]=:I+II+III.
\end{split}\end{equation}
Theorem \ref{the2.4} and \eqref{eq7} yield that $III\rightarrow 0$, as $n\rightarrow\infty$. Note that $Y-y^1\in S^\alpha_G(0,T)$ and $Y_T-y^1_T=0$. By Lemma \ref{the3.7}, we obtain $II\rightarrow 0$, as $\delta\rightarrow 0$. Then by applying Theorem \ref{the2.4} agian and \eqref{eq9}, we derive that
\begin{displaymath}
I\leq C\{\hat{\mathbb{E}}[\sup_{t\in[0,T]}\sup_{s\in[t,t+\varepsilon]}|Y_s-Y_t|^\beta]+(\hat{\mathbb{E}}[\sup_{t\in[0,T]}\sup_{s\in[t,t+\varepsilon]}|Y_s-Y_t|^\beta])^{\alpha/\beta}\}, \textrm{ as }
n\rightarrow \infty.
\end{displaymath}
First let $n\rightarrow \infty$ and then send $\varepsilon$, $\delta\rightarrow 0$ in \eqref{eq10}. The above analysis proves that for $1<\alpha<\beta$,
\begin{displaymath}\lim_{n\rightarrow\infty}\hat{\mathbb{E}}[\sup_{t\in[0,T]}|Y_t-y_t^n|^\alpha]=0.
\end{displaymath}
\end{proof}


\begin{lemma}\label{lem-the3.3} The sequence  $\{y^n,z^n,A^n\}_{n=1}^\infty$ of the solutions of $G$-BSDE (\ref{eq5}) satisfies the following properties:
\begin{align}
&\lim_{m,n\rightarrow \infty}\hat{\mathbb{E}}[\sup_{t\in[0,T]}|y_t^n-y^m_t|^\alpha]=0,\,\,\,\,\,  \textrm{ for }\,\, 1<\alpha<\beta,  \label{eqnx10}\\
&\lim_{m,n\rightarrow \infty}\hat{\mathbb{E}}[\int_0^T |z^m_s-z^n_s|^2ds]=0,\,\,\,\,\, \lim_{m,n\rightarrow \infty} \hat{\mathbb{E}}[\sup_{t\in[0,T]}|A_t^m-A_t^n|^2]=0,\label{eqnx11}
\end{align}
where we set  $A_t^n=n\int_0^t(Y_s-y_s^n)ds -K_t^n$.
\end{lemma}


\begin{proof}
By Lemma \ref{the3.8},  it is easy to check (\ref{eqnx10}).  Set $\hat{y}_t=y_t^n-y_t^m,\ \hat{z}_t=z_t^n-z_t^m,\ \hat{K}_t=K_t^n-K_t^m,\ \hat{L}_t=L_t^n-L_t^m$ and $\hat{g}_t=g(t,y_t^n,z_t^n)-g(t,y_t^m,z_t^m)$. Applying It\^{o}'s formula to $|\hat{y}_t|^2$, we get
\begin{equation}\label{e1}
\begin{split}
&\quad|\hat{y}_t|^2+\int_t^T|\hat{z}_s|^2 d\langle B\rangle_s\\
&=\int_t^T2\hat{y}_s\hat{g}_sds-\int_t^T 2\hat{y}_sd\hat{K}_s
+\int_t^T2\hat{y}_sd\hat{L}_s-\int_t^T 2\hat{y}_s\hat{z}_s dB_s\\
&\leq 2L\int_t^T|\hat{y}_s|^2+|\hat{y}_s||\hat{z}_s|ds-\int_t^T 2\hat{y}_sd\hat{K}_s
+\int_t^T2\hat{y}_sd\hat{L}_s-\int_t^T 2\hat{y}_s\hat{z}_s dB_s.
\end{split}
\end{equation}
Note that for each $\varepsilon>0$,
\[2L\int_t^T |\hat{y}_s||\hat{z}_s|ds\leq L^2/\varepsilon\int_t^T|\hat{y}_s|^2ds+\varepsilon\int_t^T|\hat{z}_s|^2ds .\]
By simple calculation, we have
\begin{align*}
\int_t^T\hat{y}_sd\hat{L}_s&=\int_t^T (y_s^n-y_s^m)[n(Y_s-y_s^n)-m(Y_s-y_s^m)]ds\\
&\leq \int_t^T (m+n)(Y_s-y_s^m)(Y_s-y_s^n)ds.
\end{align*}
Choosing $\varepsilon<\underline{\sigma}^2$ and taking expectations on both sides of \eqref{e1}, we get
\begin{displaymath}
\begin{split}
\hat{\mathbb{E}}[\int_0^T |\hat{z}_s|^2ds]&\leq C\hat{\mathbb{E}}[\int_0^T (m+n)(Y_s-y_s^n)(Y_s-y_s^m)ds
+\int_0^T|\hat{y}_s|^2ds-\int_0^T \hat{y}_sd\hat{K}_s]\\
&\leq C\hat{\mathbb{E}}[\sup_{s\in[0,T]}|Y_s-y_s^n||L_T^m|+\sup_{s\in[0,T]}|Y_s-y_s^m||L_T^n|\\
&\quad+\int_0^T|\hat{y}_s|^2ds+\sup_{t\in[0,T]}|\hat{y}_s|(|K_T^n|+|K_T^m|)].
\end{split}
\end{displaymath}
By Lemma \ref{the3.6} and Lemma \ref{the3.8}, we obtain the first convergence of (\ref{eqnx11}). For the second one, we observe that, for each $n$,   the process $A^n_t$ is  nondecreasing in $t$, and
\begin{displaymath}
A_t^m-A_t^n=(y_0^m-y_0^n)-(y_t^m-y_t^n)+\int_0^t(z_s^m-z_s^n)dB_s-\int_0^t (g(s,y_s^m,z_s^m)-g(s,y_s^n,z^n_s))ds.
\end{displaymath}
It follows from the generalized Burkholder-Davis-Gundy inequality in Proposition \ref{the1.3} and H\"{o}lder's inequality that
\begin{displaymath}
\hat{\mathbb{E}}[\sup_{t\in[0,T]}|A_t^n-A_t^m|^2]\leq C(\hat{\mathbb{E}}[\sup_{t\in[0,T]}|y_t^n-y_t^m|^2]+\hat{\mathbb{E}}[\int_0^T(z_s^n-z_s^m)^2 ds])\rightarrow 0.
\end{displaymath}
\end{proof}

We are now in the final position to prove Theorem \ref{the3.3}:

\begin{proof}[Proof of Theorem \ref{the3.3}]
From (\ref{eqnx10}) and (\ref{eqnx11}),  the sequences of $\{y^n\}_{n=1}^\infty $ converges to $Y\in S^\alpha_G(0,T)$, $\{z^n\}$ converges to a process $Z\in M_G^2(0,T)$ and  $\{A^n\}$ converges to a nondecreasing process $A\in S^2_G(0,T)$.  Thus we obtain the decomposition \eqref{eq4} by letting  $n\rightarrow\infty$ in \eqref{eq5}.

To prove  the uniqueness, let  $Z, Z' \in M_G^2(0,T)$ and $A, A'\in S^2_G(0,T)$ be such that
\begin{displaymath}
Y_t=Y_0-\int_0^t g(s,Y_s,Z_s)ds+\int_0^t Z_sdB_s- A_t=Y_0-\int_0^tg(s,Y_s,Z'_s)ds+\int_0^t Z'_sdB_s- A'_t,
\end{displaymath}
where $A,A'$ are nondecreasing processes with $A_0=A'_0=0$. By applying It\^{o}'s formula to $(Y_t-Y_t)^2(\equiv0)$ on $[0,T]$ and taking expectation, we get
\begin{displaymath}
\hat{\mathbb{E}}[\int_0^T (Z_s-Z'_s)^2d\langle B\rangle_s]=0.
\end{displaymath}
Therefore $Z_t\equiv Z'_t$. From this it follows that $A_t\equiv A'_t$.
\end{proof}

\begin{remark}
If $g=0$, then the  $\mathbb{\hat{E}}^{g}$-supermartingale $\{Y_t\}_{t\in[0,T]}$ is a $G$-supermartingale. Theorem \ref{the3.3} also holds for this special case which is in fact the Doob-Meyer decomposition theorem for $G$-supermartinales. The penalized $G$-BSDEs is of the following form, $n=1,2,\cdots$,
\begin{displaymath}
y_t^n=Y_T+n\int_t^T(Y_s-y_s^n)ds-\int_t^T Z_s^ndB_s-(K_T^n-K_t^n).
\end{displaymath}

We can show Lemma \ref{the3.5} in a simple way. Since the above $G$-BSDE is linear, we can solve it explicitly by applying It\^{o}'s formula to $e^{-nt}y_t^n$,
\begin{displaymath}
e^{-nt}y_t^n+\int_t^T e^{-ns}Z_s^n dB_s+\int_t^T e^{-ns}dK_s^n=e^{-nT}+\int_t^T ne^{-ns}Y_s^nds.
\end{displaymath}
According to Lemma 3.4 in \cite{HJPS}, $\{\int_0^t e^{-ns}dK_s^n\}_{t\in[0,T]}$ is a $G$-martingle. Thus we get
\begin{displaymath}
\begin{split}
y_t^n&=e^{nt}\hat{\mathbb{E}}_t[e^{-nT}Y_T+\int_t^T ne^{-ns}Y_s ds]\\
&\leq e^{n(t-T)}\hat{\mathbb{E}}_t[Y_T]+\int_t^T ne^{n(t-s)}\hat{\mathbb{E}}_t[Y_s]ds\\
&\leq e^{n(t-T)}Y_t+\int_t^T ne^{n(t-s)}Y_t ds=Y_t.\\
\end{split}
\end{displaymath}
Furthermore, if $g$ is a linear function, the proof is similar.
\end{remark}

\begin{remark}
By Theorem 4.5 in \cite{S11} (see also Theorem 5.1 in \cite{STZ11}), for a $G$-martingale $X_t=\mathbb{\hat{E}}_t[\xi]$, $t\in[0,T]$, where $\xi\in L_G^\beta(\Omega_T)$ with $\beta>1$, we have
\[
X_t=X_0+\int_0^t Z_sdB_s+K_t,
\]
here $\{K_t\}$ is a decreasing $G$-martingale. Similar to the classical case, given a $G$-supermartingale $Y$, one may conjecture that
\begin{equation}\label{eq11}
Y_t=Y_0+\int_0^t Z_s dB_s +K_t-L_t, 
\end{equation}
where $\{K_t\}$ is a decreasing $G$-martingale and $\{L_t\}$ is a nondecreasing process with $L_0=0$.  The problem is that the above representation is not unique unless $K\equiv0$: $\widetilde{K}\equiv0$, $\widetilde{L}=L-K$ is a different decomposition. That is why we put the increasing process $\{L_t-K_t\}$ as an integral.

It is worth pointing out that unlike with the classical case, considering the decomposition theorem for $G$-submartingales is fundamentally different from that for $G$-supermartingales. Indeed, if Y is a $G$-submartingale, $\{L_t\}$ in \eqref{eq11} should be a nonincreasing process. Therefore $\{L_t-K_t\}$ ends up with a finite variation process. Then this situation becomes much more complicated. We would like to refer the reader to \cite{PZ} which defines a new norm for $G$-submartingales. As a byproduct, the decomposition is unique.
\end{remark}



Then we establish the decomposition theorem for $\mathbb{\hat{E}}^{g,f}$-supermartingales.
\begin{theorem}\label{the3.9}
Let $Y=(Y_t)_{t\in[0,T]}\in S_G^\beta(0,T)$ be an $\mathbb{\hat{E}}^{g,f}$-supermartingale under with $\beta>2$. Suppose that $f$ and $g$ satisfy (H1') and (H2). Then $(Y_t)$ has the following decomposition
\begin{equation}
Y_t=Y_0-\int_0^t g(s,Y_s,Z_s) ds-\int_0^t f(s,Y_s,Z_s)d\langle B\rangle_s+\int_0^t Z_s dB_s-A_t,\quad \textrm{q.s.},
\end{equation}
where $\{Z_t\}\in M_G^2(0,T)$ and $\{A_t\}$ is a continuous nondecreasing process with $A_0=0$ and $A_T\in L_G^2(\Omega_T)$. Furthermore, the above decomposition is unique.
\end{theorem}


\section{$\mathbb{\hat{E}}^g$-supermartingales and related PDEs}
In this section, we present the relationship between the $\mathbb{\hat{E}}^g$-supermartingales and the fully nonlinear parabolic PDEs. For this purpose, we will put the $\mathbb{\hat{E}}^g$-supermartingales in a Markovian framework.

We will make the following assumptions throughout this section. Let $b,\ h,\ \sigma:[0,T]\times\mathbb{R}\rightarrow\mathbb{R}$
 and $g:[0,T]\times\mathbb{R}^3\rightarrow\mathbb{R}$
 be deterministic functions and satisfy the following conditions:
\begin{description}
\item[(H4.1)] $b,\ h,\ \sigma, \ g$ are continuous in $t$;
\item[(H4.2)] There exists a constant $L>0$, such that
\begin{align*}
&|b(t,x)-b(t,x')|+|h(t,x)-h(t,x')|+|\sigma(t,x)-\sigma(t,x')|\leq L|x-x'|,\\
&|g(t,x,y,z)-g(t,x',y',z')|\leq L(|x-x'|+|y-y'|+|z-z'|).
\end{align*}
\end{description}

For each $t\in[0,T]$ and $\xi\in L_G^2(\Omega_t)$, we consider the following type of SDE driven by $1$-dimensional $G$-Brownian motion:
\begin{equation}\label{eq1.10}
dX_{s}^{t,\xi}=b(s,X_{s}^{t,\xi})ds+h(s,X_{s}^{t,\xi})d\langle B\rangle_s+\sigma(s,X_{s}^{t,\xi})dB_s, \quad X_t^{t,\xi}=\xi.
\end{equation}
 We have the following estimates which can be found in Chapter V in \cite{P10}.
\begin{proposition}\label{the1.12}
Let $\xi,\xi'\in L_G^p(\Omega_t,\mathbb{R})$ with $p\geq2$. Then we have, for each $\delta\in[0,T-t]$
\begin{align*}
\hat{\mathbb{E}}_t[|X_{t+\delta}^{t,\xi}-X_{t+\delta}^{t,\xi'}|^p]&\leq C|\xi-\xi'|^p,\\
\hat{\mathbb{E}}_t[\sup_{s\in[t,t+\delta]}|X_s^{t,\xi}-\xi|^p]&\leq C(1+|\xi|^p)\delta^{p/2},\\
\hat{\mathbb{E}}_t[\sup_{s\in[t,T]}|X_s^{t,\xi}|^p]&\leq C(1+|\xi|^p),
\end{align*}
where the constant $C$ depends on $L,\ p$, $T$ and the function $G$. Consequently, for each $(x,y,z)\in \mathbb{R}^3$ and $p\geq 2$, we have $\{X^{t,x}_s\}_{s\in[t,T]},\ \{g(s,X_s^{t,x},y,z)\}_{s\in[t,T]}\in M_G^{p}(0,T)$.
\end{proposition}




Consider the following type of PDE:
\begin{equation}\label{eq1.12}
\partial_t u+F(D_x^2u,D_xu,u,x,t)=0,
\end{equation}
where
 \begin{align*}
F(D_x^2 u,D_x u,u,x,t)=&G(H(D_x^2 u,D_x u,u,x,t))+b(t,x)D_x u
+g(t,x,u,\sigma(t,x)D_x u),\\
H(D_x^2 u,D_x u,u,x,t)=&\sigma^2(t,x)D_x^2u+2 h(t,x) D_xu.
\end{align*}

Now we shall recall the definition of viscosity solution to equation \eqref{eq1.12}, which is introduced in \cite{CIL}. Let $u\in C((0,T)\times\mathbb{R})$ and $(t,x)\in(0,T)\times\mathbb{R}$. Denote by $\mathcal{P}^{2,-}u(t,x)$ (the ``parabolic subjet" of $u$ at $(t,x)$) the set of triples $(a,p,X)\in\mathbb{R}^3$ such that
\[
u(s,y)   \geq u(t,x)+a(s-t)+ p(y-x)+\frac{1}{2} X(y-x)^2+o(|s-t|+|y-x|^{2}).
\]
Similarly, we define $\mathcal{P}^{2,+}u(t,x)$ (the ``parabolic superjet" of $u$ at $(t,x)$) by $\mathcal{P}^{2,+}u(t,x):=-\mathcal{P}^{2,-}(-u)(t,x)$.

\begin{definition}
$u\in C((0,T)\times \mathbb{R})$ is called a viscosity supersolution (resp. subsolution) of \eqref{eq1.12} on $(0,T)\times \mathbb{R}$ if at any point $(t,x)\in(0,T)\times \mathbb{R}$, for any $(a,p,X)\in \mathcal{P}^{2,-}u(t,x)$ (resp. $\in \mathcal{P}^{2,+}u(t,x)$)
\[
a+F(X,p,u,x,t)\leq 0 \ (resp. \geq 0).
\]
$u\in C((0,T)\times \mathbb{R})$ is said to be a viscosity solution of \eqref{eq1.12} if it is both a viscosity supersolution and a viscosity subsolution.
\end{definition}


\begin{remark}
We then give the following equivalent definition (see \cite{CIL}). $u\in C((0,T)\times\mathbb{R})$ is called a viscosity supersolution (resp. subsolution) of \eqref{eq1.12} on $(0,T)\times\mathbb{R}$ if for each fixed  $(t,x)\in(0,T)\times\mathbb{R}$, $v\in C^{1,2}((0,T)\times\mathbb{R})$ such that $u(t,x)=v(t,x)$ and $v\leq u$ (resp. $v\geq u$) on $(0,T)\times\mathbb{R}$, we have
\[\partial_t v(t,x)+F(D_x^2v(t,x),D_xv(t,x),v(t,x),x,t)\leq 0 \ (resp. \geq 0).\]
\end{remark}

We state the main result of this section.
\begin{theorem}\label{the1.13}
Assume (H4.1) and (H4.2) hold. Let $u:[0,T]\times\mathbb{R}\rightarrow\mathbb{R}$ be uniformly continuous with respect to $(t,x)$ and satisfy
\[|u(t,x)|\leq C(1+|x|^k),\ \ t\in[0,T], x\in\mathbb{R},\] where $k$ is a positive integer. Then $u$ is  a  viscosity supersolution of equation \eqref{eq1.12},  if and only if  $\{Y^{t,x}_s\}_{s\in[t,T]}:=\{u(s,X_s^{t,x})\}_{s\in[t,T]}$ is an $\mathbb{\hat{E}}^{g^{t,x}}$-supermartingale, for each fixed $(t,x)\in(0,T)\times\mathbb{R}$, where $g^{t,x}=g(s,X_s^{t,x},y,z)$ and $\{X^{t,x}_s\}_{s\in[t,T]}$ is given by \eqref{eq1.10}.
\end{theorem}

To prove this theorem, we introduce the following lemma.
\begin{lemma}\label{the1.14}
We have, for each $p>2$ and $(t,x)\in[0,T)\times\mathbb{R}$, $\{Y_s^{t,x}\}_{s\in[t,T]}\in S_G^p(0,T)$.
\end{lemma}
\begin{proof}
Note that
\begin{align*}
\sup_{s\in[t,T]}|Y_s^{t,x}|^{p}\leq C\sup_{s\in[t,T]}(1+|X_s^{t,x}|^{kp}).
\end{align*}
By Proposition \ref{the1.12}, we have $\mathbb{\hat{E}}[\sup_{s\in[t,T]}|Y_s^{t,x}|^{p}]<\infty$. Since $u$ is uniformly continuous, we get the desired result.
\end{proof}


\medskip

\begin{proof}[Proof of Theorem \ref{the1.13}]
For a given function  $u$ satisfying the conditions in Theorem \ref{the1.13} and for each $n=1,2,\cdots$, $(t,x)\in [0,T)\times{\mathbb R}$,
 let us  consider the following  $G$-BSDEs:
\begin{displaymath}\label{eq1.13}
y_s^{n,t,x}=Y_T^{t,x}+\int_s^T g(r,X^{t,x}_r,y_r^{n,t,x},z_r^{n,t,x})dr+n\int_s^T(Y_r^{t,x}-y_r^{n,t,x})dr-\int_s^T z_r^{n,t,x}dB_r-(K_T^{n,t,x}-K_s^{n,t,x}),
\end{displaymath}
and, correspondingly,  the following viscosity solution of PDEs:
\begin{displaymath}\label{eq1.14}
\partial_t v^n(t,x)+F(D_x^2 v^n(t,x),D_x v^n(t,x),v^n(t,x),x,t)+n(u(t,x)-v^n(t,x))=0,
\end{displaymath}
defined on $(0,T)\times {\mathbb{R}}$ with the Cauchy condition
\begin{displaymath}\label{eq1.14a}
v^n(T,x)=u(T,x).
\end{displaymath}
From  the nonlinear Feynman-Kac formula obtained in \cite{HJPS1} (i.e., Theorem 4.5 in \cite{HJPS1}), it follows that  $y_s^{n,t,x}=v^n(s,X_s^{t,x})$, $s\in [t,T]$.

To prove the ``if" part of the Theorem, we assume that, for each $(t,x)$, $\{Y_{\cdot}^{t,x}\}$ is an $\hat{\mathbb{E}}^{g^{t,x}}$-supermartingale on $[t,T]$.
 Observing that $(y^{n,t,x}_s,z^{n,t,x}_s,K^{n,t,x}_s)_{s\in [t,T]}$ is a special case of (\ref{eq5}),  we can apply Lemma \ref{the3.5} and  Lemma \ref{lem-the3.3} to prove that
$y^{n,t,x}_s\leq Y^{t,x}_s$ and then to get the convergence of $\{y^{n,t,x}_\cdot\}$ to $\{Y^{t,x}_\cdot\}$ on $[t,T]$, similar to (\ref{eqnx10}).
 By the proof of Theorem \ref{the3.3}, for any $(t,x)\in[0,T]\times\mathbb{R}$, we have
\[
v^n(s,X^{t,x}_s)=y^{n,t,x}_s\leq Y^{t,x}_s=u(s,X^{t,x}_s),
\]
and $v^{n}\uparrow u$. Since $u$ is uniformly continuous on $[0,T]\times {\mathbb{R}}$, the convergence is also locally uniform. By Theorem 4.5 in \cite{HJPS1} and noting that $v^n\leq u$, $v^n$ is a viscosity supersolution of PDE \eqref{eq1.12}. It follows from the stability theorem of the viscosity solutions (see Proposition 4.3 in \cite{CIL}) that the limit function $u$ is also a viscosity supersolution of  PDE \eqref{eq1.12}.

 Now we prove the ``only if" part of the Theorem. 
 For each $t_1\in(0,T)$, let $v^{t_1,u(t_1,\cdot)}$ be the viscosity solution of PDE \eqref{eq1.12} on $(0,t_1)\times\mathbb{R}$ with Cauchy condition $v^{t_1,u(t_1,\cdot)}(t_1,x)=u(t_1,x)$. By the comparison theorem for viscosity solutions, for each $(s,x)\in[0,t_1]\times\mathbb{R}$, it is easy to check that $v^{t_1,u(t_1,\cdot)}(s,x)\leq u(s,x)$. For any $t\leq s\leq r\leq T$, by the nonlinear Feymann-Kac formula in \cite{HJPS1}, we have
\begin{equation}\label{eq1.15}
\hat{\mathbb{E}}^{g^{t,x}}_{s,r}[Y_r^{t,x}]=\hat{\mathbb{E}}^{g^{t,x}}_{s,r}[u(r,X_r^{t,x})]= \hat{\mathbb{E}}^{g^{t,x}}_{s,r}[v^{r,u(r,\cdot)}(r,X_r^{t,x})]=v^{r,u(r,\cdot)}(s,X_s^{t,x})\leq u(s,X_s^{t,x})= Y_s^{t,x},
\end{equation}
which implies that $\{Y^{t,x}_s\}_{s\in[t,T]}:=\{u(s,X_s^{t,x})\}_{s\in[t,T]}$ is an $\mathbb{\hat{E}}^{g^{t,x}}$-supermartingale. The proof is complete.
\end{proof}
\medskip

 The following result can be considered as the ``inverse" comparison theorem for viscosity solutions of PDEs.
\begin{corollary} Let $V:[0,T]\times\mathbb{R}\rightarrow \mathbb{R}$ be uniformly continuous with respect to $(t,x)$ and  satisfy
\[|V(t,x)|\leq C(1+|x|^k),\ \ t\in[0,T], x\in\mathbb{R},\] where $k$ is a positive integer. Assume that
\[
V(t,x)\geq u^{t_1,V(t_1,\cdot)}(t,x),\,\,\,\,\, \forall (t,x)\in [0,t_1]\times {\mathbb{R}}, \,\,\, t_1\in [0,T],
\]
where $u^{t_1,V(t_1,\cdot)}$ denotes the viscosity solution of PDE (\ref{eq1.12}) on $(0,t_1)\times \mathbb{R}$ with Cauchy condition $u^{t_1,V(t_1,\cdot)}(t_1,x)=V(t_1,x)$. Then $V$ is a viscosity supersolution of  PDE (\ref{eq1.12}) on  $(0,T)\times {\mathbb{R}}$.
\end{corollary}
\begin{proof}  For each fixed $(t,x)$, set $\{Y^{t,x}_s\}_{s\in[t,T]}:=\{V(s,X_s^{t,x})\}_{s\in[t,T]}$. 
Similar with \eqref{eq1.15}, $\{Y^{t,x}_s\}_{s\in[t,T]}$ is an $\mathbb{\hat{E}}^{g^{t,x}}$-supermartingale. It follows from Theorem \ref{the1.13} that $V$ is a viscosity supersolution of PDE  (\ref{eq1.12}).
\end{proof}

\textbf{Conclusion}
We obtain the decomposition theorem of Doob-Meyer's type for  $\mathbb{\hat{E}}^{g}$-supermartingales, which is a generalization of the results of Peng \cite{P99}. Our theorem provides the first step for solving the representation theorem of dynamically consistent nonlinear expectations. Different from the classical case, the decomposition theorem for $\mathbb{\hat{E}}^g$-submartingales remains open.





\end{document}